\documentclass[12pt,a4paper]{article}
\usepackage{amssymb,amsmath,amsthm}

\newcommand\cF{{\mathcal F}}

\newcommand\cH{{\mathcal H}}

\newcommand\cP{{\mathcal P}}

\theoremstyle{plain}
\newtheorem{theorem}{Theorem}[section]
\newtheorem{lemma}[theorem]{Lemma}

\newtheorem{conjecture}[theorem]{Conjecture}

\theoremstyle{definition}

\newcommand\lref[1]{Lemma~\ref{lem:#1}}
\newcommand\tref[1]{Theorem~\ref{thm:#1}}
\newcommand\cref[1]{Corollary~\ref{cor:#1}}

\newcommand\cjref[1]{Conjecture~\ref{conj:#1}}
\newcommand\sref[1]{Section~\ref{sec:#1}}

\title{A note on traces of set families}

\author{Bal\'azs Patk\'os\thanks{Alfr\'ed R\'enyi Institute of Mathematics, P.O.B. 127, Budapest H-1364, Hungary. Email: patkos@renyi.hu. Research supported by
    Hungarian National Scientific Fund, grant number: PD-83586 and the J\'anos Bolyai Research Scholarship of the Hungarian Academy of Sciences}}

\begin{document}
\maketitle
\begin{abstract}
A family of sets $\cF \subseteq 2^{[n]}$ is defined to be $l$-trace
$k$-Sperner if for any $l$-subset $L$ of $[n]$ the family of traces
$\cF|_L=\{F \cap L: F \in \cF\}$ does not contain any
chain of length $k+1$. In this paper we prove that for any positive integers $l',k$ with $l'<k$ if $\cF$ is $(n-l')$-trace $k$-Sperner, then $|\cF| \le (k-l'+o(1))\binom{n}{\lfloor n/2\rfloor}$ and this bound is asymptotically tight.
\end{abstract}

\textit{AMS subject classification}: 05D05

\textit{Keywords}: traces of set families, $k$-Sperner families, forbidden configurations

\vskip 0.5truecm


\section{Introduction}
We use standard notation. The set of the first $n$ positive integers is denoted by $[n]$. For a set $X$ the family of all subsets of $X$, all $i$-subsets of $X$, all subsets of $S$ of size at most $i$, all subsets of $S$ of size at least $i$ are denoted by $2^X, \binom{X}{i}, \binom{X}{\le i}, \binom{X}{\ge i}$, respectively.

\vskip 0.3truecm

Probably the very first theorem in extremal finite set theory is Sperner's result \cite{S} stating that if a family $\cF \subseteq 2^{[n]}$ does not contain two sets $F_1,F_2$ with $F_1 \subset F_2$, then the size of $\cF$ cannot exceed $\binom{n}{\lfloor n/2 \rfloor}$. Moreover, the only families attaining this size are $\binom{[n]}{\lfloor n/2 \rfloor}$ and, if $n$ is odd, $\binom{[n]}{\lceil n/2 \rceil}$. This theorem was generalized by Erd\H os \cite{E} in the following way: if a family $\cF \subseteq 2^{[n]}$ does not contain any chain $F_1 \subset F_2 \subset ... \subset F_k \subset F_{k+1}$ of length $k+1$ (families with this property are called \textit{$k$-Sperner families}), then the size of $\cF$ cannot exceed $\sum_{i=1}^k\binom{n}{\lfloor \frac{n-k}{2} \rfloor+i}$.

Another topic in extremal finite set theory deals with problems concerning \textit{traces} of set families. The trace of a set $F$ on another set $X$ is $F|_X=F \cap X$, while the trace of a family $\cF$ is $\cF|_X=\{F|_X: F\in\cF\}$. The fundamental theorem about traces is the so-called Sauer-lemma \cite{Sa,Sh,VC} that states that if $\cF\subseteq 2^{[n]}$ contains more than $\sum_{i=0}^{l-1}\binom{n}{i}$ sets, then there exists an $L \in \binom{[n]}{l}$ such that $\cF|_L=2^L$. As opposed to the situation described in Erd\H os's theorem, there are lots of different extremal families (see e.g. \cite{FQ}). In \cite{P}, the present author showed that $\binom{[n]}{\le l-1}$ and $\binom{[n]}{\ge n-l+1}$ are the only families $\cF$ of size $\sum_{i=0}^{l-1}\binom{n}{i}$ such that for all $L \in \binom{[n]}{l}$ the trace $\cF|_L$ does not contain any chain of length $l+1$ (i.e. maximal chains in $2^L$). This result led to the following definition: a family $\cF$ is said to be $l$-trace $k$-Sperner if for any $l$-set $L$ the trace $\cF|_L$ is $k$-Sperner. Let $f(n,k,l)$ denote the maximum size that an $l$-trace $k$-Sperner family $\cF \subseteq 2^{[n]}$ can have. In \cite{P}, it was also shown that for any pair of integers $k,l$ there exists $n_0(k,l)$ such that if $n \ge n_0$, then $f(n,k,l)=\sum_{i=0}^{k-1}\binom{n}{i}$.

The situation becomes totally different when one considers the problem of determining $f(n,k,n-l')$ with $k,l'$ fixed and $n$ large enough. Note that if $a \le |A| \le b$ holds, then for any $l'$-subset $L$ the size of $A|_{[n]\setminus L}$ lies between $a-l'$ and $b$. Therefore, as a chain contains sets of different sizes, the family $\bigcup_{i=1}^{k-l'}\binom{[n]}{a+i}$ is $(n-l')$-trace $k$-Sperner for any values of $a,k,l'$ and $n$. The following conjecture asserts that the largest $(n-l')$-trace $k$-Sperner family is of this sort if $n$ is large enough.

\begin{conjecture}
\label{conj:precise} Let $k$ and $l'$ be positive integers with $l'<k$. Then there exists $n_0=n_0(k,l')$ such that if $n \ge n_0$ and $\cF \subseteq 2^{[n]}$ is an $(n-l')$-trace $k$-Sperner family, then $|\cF| \le \sum_{i=1}^{k-l'}\binom{n}{\lfloor\frac{n-(k-l')}{2}+ i\rfloor}$. 
\end{conjecture}

Note that if true, the bound in \cjref{precise} is best possible as shown by the family $\bigcup_{i=1}^{k-l'}\binom{[n]}{\lfloor\frac{n-(k-l')}{2}+ i\rfloor}$. In \cite{P} it was shown that \cjref{precise} holds asymptotically when $l'=1,k=2$. The main result of this paper verifies \cjref{precise} asymptotically for all values of $k$ and $l'$.

\begin{theorem}
\label{thm:main} Let $k$ and $l'$ be positive integers with $l'<k$. Then if $\cF \subseteq 2^{[n]}$ is an $(n-l')$-trace $k$-Sperner family, then $|\cF| \le (k-l'+o(1))\binom{n}{\lfloor n/2\rfloor}$.
\end{theorem}

The rest of the paper is organized as follows: in \sref{lanp}, we briefly summarize the problem of forbidden subposets in set families (for recent survey-like papers see \cite{K1,K2} and for the most recent results see \cite{GLL}) and state a result of Bukh \cite{B} that will be used in the proof of \tref{main}. In \sref{proof}, we obtain a result about $f(n,l',n-l')$ and another one about the connection of $f(n,l',n-l')$ and $f(n,k,n-l')$. These two results will immediately imply \tref{main}. \sref{rem} contains some concluding remarks and open problems.

\section{Families with forbidden subposets}
\label{sec:lanp}

The aim of this section is to describe the context of forbidden subposets, introduce some terminology and to state \tref{bukh} that will serve as the main tool in proving \tref{main}.

We say that a family $\cF$ of sets contains a poset $P$ if there is an injective mapping $i: P \rightarrow \cF$ such that whenever $p \leqslant_P q$ holds, then $i(p)$ is contained in $i(q)$. We say that $\cF$ is \textit{$P$-free} if it does not contain $P$.
For any set
$\cP$ of posets $La(n,\cP)$ denotes the maximum size that a family $\cF
\subseteq 2^{[n]}$ can have such that $\cF$ is $P$-free for all $P \in
\cP$. If $\cP$ consists of a single poset $P$, we write $La(n,P)$ instead of $La(n,\{P\})$.
With this notation Sperner's theorem determines $La(n,P_2)$ and Erd\H os's theorem determines $La(n,P_{k+1})$, where $P_k$ denotes the poset consisting of a chain of length $k$. In these theorems, $La(n,P_k)$ is attained at a union of consecutive levels of $2^{[n]}$. It is natural to conjecture that something similar is true for all posets. For a poset $P$ let $l(P)$ denote the largest integer $l$ such that for any $n$, no $l$ consecutive levels of $2^{[n]}$ contain $P$. The following conjecture is folklore.

\begin{conjecture}
\label{conj:poset}
Let $P$ be a finite poset. Then $La(n,P)=(l(P)+\frac{1}{n})\binom{n}{\lfloor n/2 \rfloor }$.
\end{conjecture}

The \textit{Hasse graph} $H(P)$ of a poset $P$ is a directed graph 
with vertex set $P$ and $(p,q)$ is an arc if and only if $p \prec_P q$ (i.e. $p <_P q$ and there does not exist $r \in P$ with $p <_P r <_P q$).
The height $h(P)$ of a poset is the length of the longest chain in $P$. It is easy to verify that if $H(P)$ is a tree, then $l(P)=h(P)-1$. \cjref{poset} was proved by Bukh for all posets $P$ with $H(P)$ being a tree.

\begin{theorem} [Bukh, \cite{B}]
\label{thm:bukh}
Let $P$ be a finite poset such that $H(P)$ is a tree. Then $La(n,P)=(h(P)-1+O(\frac{1}{n}))\binom{n}{\lfloor n/2 \rfloor }$.
\end{theorem}

\section{Proof of \tref{main}}
\label{sec:proof}

In this section we prove \tref{main}. To be able to use \tref{bukh}, we need to define the following directed graph: $T_{h,c}$ is a tree with height $h$ such that all arcs are directed towards the root and each vertex, with the exception of the leaves, has exactly $c$ children. Let $P_{h,c}$ denote the poset with $H(P_{h,c})=T_{h,c}$. The following two theorems immediately yield \tref{main}.

\begin{theorem}
\label{thm:induction} Let $k,l'$ be positive integers with $l'<k$. Then the following inequality holds:
\[
f(n,k,n-l') \le f(n,l',n-l')+La(n,P_{k-l'+1,2^{l'}}).
\] 
\end{theorem}

\begin{theorem}
\label{thm:weakbase} For any positive integer $l'$, the size of an $(n-l')$-trace $l'$-Sperner family $\cF \subseteq 2^{[n]}$ is $O_{l'}(n^{-1/3}\binom{n}{\lfloor n/2\rfloor})$.
\end{theorem}

\begin{proof}[Proof of \tref{induction}] Let $\cF \subseteq 2^{[n]}$ be a set family of size $f(n,l',n-l')+La(n,P_{k-l'+1,2^{l'}})+1$. We will find an $l'$-subset $L\subset [n]$ and a chain of length $k+1$ in $\cF|_{[n]\setminus L}$. By the size of $\cF$, there exists a copy of $P_{k-l'+1,2^{l'}}$ in $\cF$. We remove the set corresponding to the root of $T_{k-l'+1,2^{l'}}$ and repeat this procedure until there exists no more copy of $P_{k-l'+1,2^{l'}}$ in the remaining family. As $|\cF|= f(n,l',n-l')+La(n,P_{k-l'+1,2^{l'}})+1$, we must have removed at least $f(n,l',n-l')+1$ sets. Thus, there exists an $l'$-subset $L \subseteq [n]$ and $l'+1$ removed sets $F_{k-l'+1},F_{k-l'+2},...,F_{k},F_{k+1}$ such that 
\[
F_{k-l'+1}|_{[n]\setminus L}\subsetneq F_{k-l'+2}|_{[n]\setminus L}\subsetneq ...\subsetneq F_{k}|_{[n]\setminus L}\subsetneq F_{k+1}|_{[n]\setminus L}
\]
holds.

As $F_{k-l'+1}$ is a removed set, there exists a copy of $P_{k-l'+1,2^{l'}}$ such that $F_{k-l'+1}$ corresponds to its largest element. Therefore there are lots of chains of length $k-l'$ in $\cF$ such that all of their elements are subsets of $F_{k-l'+1}$. Clearly, if $G \subseteq G'$, then $G|_{[n]\setminus L} \subseteq G'|_{[n]\setminus L}$, but we also require the sets of the chain not to coincide when considering their traces on $[n]-L$. Thus, we need a chain $F_1\subsetneq F_2 \subsetneq ...\subsetneq F_{k-l'}\subsetneq F_{k-l'+1}$ such that $F_{i+1} \setminus F_i$ is not contained in $L$ for all $i=1,...,k-l'$. 
Suppose we have already picked $F_j$ from the $j$th level of the copy of $P_{k-l'+1,2^{l'}}$ for all $j=i+1,...,k-l'+1$. Then $F_{i+1}$ has $2^{l'}$ children in $P_{k-l'+1,2^{l'}}$. As for any $F$ of these sets, we have $F_{i+1} \setminus F\neq \emptyset$, and $L$ has $2^{l'}-1$ non-empty subsets, at least one such $F$ will satisfy $F|_{[n]\setminus L}\subsetneq F_{i+1}|_{[n] \setminus L}$. 
Letting this $F$ be $F_i$ we continue to define all $F_j$'s and we get a chain of length $k+1$ in $\cF|_{[n]\setminus L}$. This shows that $\cF$ cannot be $(n-l')$-trace $k$-Sperner.
\end{proof}

\begin{proof}[Proof of \tref{weakbase}]
Let $\cF \subseteq 2^{[n]}$ be an $(n-l')$-trace $l'$-Sperner family and let $\cF_i=\{F \in \cF: |F|=i\}$ for all $i=0,1,...,n$. Note that if $\cH \subseteq \binom{[n]}{i}$ is $(n-l')$-trace $l'$-Sperner, then $\cH$ does not contain sets $H_1,H_2,...,H_{l'+1}$ such that for some $x_1,x_2,...,x_{i+l'} \in [n]$ we have $H_j=\{x_j,x_{j+1},...,x_{j+i-1}\}$ for all $j=1,2,...,l'+1$ (sets satisfying these conditions are often said to form a \textit{tight path of length $l'+1$}). Indeed, if such sets exist, then the traces of the $H_j$'s form a chain of length $l'+1$ on the set $[n] \setminus \{x_1,x_2,...,x_{l'}\}$ provided $i \ge l'$. The result we found in the literature concerning uniform families not containing tight paths of given length \cite{GKL} is not strong enough for our purposes, thus we prove the following lemma. 

\begin{lemma}
\label{lem:uniform} For any positive integer $l'$, if $\cH \subseteq \binom{[n]}{i}$ does not contain a tight path of length $l'+1$, then $|\cH|=O_{l'}(\frac{1}{i}\binom{n}{i-1})$ provided $i\ge 2l'$.
\end{lemma}

\begin{proof}
We proceed by induction on $l'$. If $l'=1$, then the above requirement is equivalent to the fact that for any $H,H' \in \cH$ the shadows $\{G\subset H: |G|=|H|-1\}$ and $\{G'\subset H': |G'|=|H'|-1\}$ are disjoint. Therefore $|\cH| \le \frac{1}{i}\binom{n}{i-1}$.

Let us assume that we have already proved the existence of a constant $c_{l'}$ such that any family $\cH \subseteq \binom{[n]}{i}$ without a tight path of length $l'$ has size at most $\frac{c_{l'}}{i}\binom{n}{i-1}$. Let us define $c_{l'+1}=c_{l'}+2(l'+1)$ and consider a family $\cH \subseteq \binom{n}{i}$ with $|\cH| \ge \frac{c_{l'+1}}{i}\binom{n}{i-1}$. By the induction hypothesis we find a tight path of length $l'$. Removing the last set of this path we can still find another tight path of length $l'$. In this way, we find $\frac{c_{l'+1}-c_{l'}}{i}\binom{n}{i-1}=\frac{2(l'+1)}{i}\binom{n}{i-1}$ different sets in $\cH$ such that each of them is the last set in a certain tight path of length $l'$. 

Let $\cH_1$ denote the subfamily of these sets and consider a set $H \in \cH_1$. Let $H'$ denote the first set of (one of) the tight path(s) to which $H$ belongs, i.e. if the vertices of the tight path are $x_1,x_2,...,x_{i+l'-1}$ and $H=\{x_{l'},x_{l'+1},...,x_{i+l'-1}\}$, then $H'=\{x_1,x_2,...,x_{i}\}$. Let the \textit{modified shadow} of $H$ with respect to $H'$ be $\{H \setminus \{x_j\}: l' \le j \le i\}$. Clearly, the size of the modified shadow determined by all tight paths is $i-l'+1\ge i/2$ by the assumption $i \ge 2l'$. Therefore, there exists an $(i-1)$-set $G$ that belongs to the modified shadows of at least $l'+1$ sets $H^1,H^2,...,H^{l'+1}$ from $\cH_1$. 

Let $P_1,P_2,...,P_{l'}=H^1$ be a tight path of length $l'$ on the vertices $\{y_1,y_2,...,y_{i+l'-1}\}$ with $\{y_j,y_{j+1},...,y_{j+i-1}\}=P_j \in \cH$ for all $j=1,2,...,l'$ and let $G=H^1 \setminus \{y_t\}$ for some $l' \le t \le i$. As the $H^j$'s are all different containing $G$ and have size $i$ at least one of them, say $H^2$, is of the form $G \cup \{z\}$ such that $z \notin \{y_1,y_2,...,y_{l'-1},y_t\}$. But then the sets $P_1,P_2,...,P_{l'}=H^1,H^2$ form a tight path of length $l'+1$ on the vertices $\{y_1,y_2,...,y_{l'-1},y_t,y_{l'},y_{l'+1},...,y_{i+l'-1},z\}$.
This finishes the proof of the induction step.
\end{proof}

It is well known that $|\{X\subseteq [n]: ||X|-n/2| \ge n^{2/3}\}| = o(\frac{1}{n}\binom{n}{\lfloor n/2\rfloor})$. Therefore by \lref{uniform} we have
\[
|\cF|=o\left(\frac{1}{n}\binom{n}{\lfloor n/2\rfloor}\right)+\sum_{i=n/2-n^{2/3}}^{n/2+n^{2/3}}|\cF_i|=2n^{2/3}O_{l'}\left(\frac{1}{n}\binom{n}{\lfloor n/2 \rfloor}\right)=O_{l'}\left(n^{-1/3}\binom{n}{\lfloor n/2 \rfloor}\right).
\]
\end{proof}

\section{Concluding remarks}
\label{sec:rem}
Let us first remark that we do not need the full strength of Bukh's theorem. An almost identical proof to that of \tref{induction} shows that the inequality $f(n,k+1,n-l') \le f(n,k,n-l')+La(n,P_{2,l'})$ holds. Thanh showed $La(n,P_{2,l'})=(1+O_{l'}(\frac{1}{n}))\binom{n}{\lfloor n/2 \rfloor}$ in an earlier paper \cite{T} and with a much easier proof than that of \tref{bukh}. (Later, De Bonis and Katona improved the error term \cite{BK}.) However, as it is very rare that the extremal family for a forbidden subposet problem consists only of full levels, it seems unlikely that \cjref{precise} could be proved using only results from that area.

\tref{main} and \cjref{precise} do not consider the case $k \le l'$. In \cite{P} it was proved that $f(n,1,n-l')=\Theta_{l'}(\frac{1}{n^{l'}}\binom{n}{\lfloor n/2 \rfloor})$. \tref{weakbase} states that $f(n,l',n-l')=O_{l'}(\frac{1}{n^{1/3}}\binom{n}{\lfloor n/2 \rfloor})$ and it is natural to conjecture that bound $O_{l'}(\frac{1}{n}\binom{n}{\lfloor n/2 \rfloor})$ holds in general, not only for uniform families as proved by \lref{uniform}. We would like to propose the following conjecture that, if true, would generalize all results and conjectures above.

\begin{conjecture}
\label{conj:othercase}
For any pair of integers $k \le l'$, the following holds
\[
f(n,k,n-l')=\Theta_{k,l'}\left(\frac{1}{n^{l'-k+1}}\binom{n}{\lfloor n/2 \rfloor}\right).
\]
\end{conjecture}

\vskip 0.5truecm

\noindent \textbf{Acknowledgment.} We would like to thank an anonymous referee for his/her careful reading.

\end{document}